\documentclass[12pt,reqno]{amsart}

\usepackage[a4paper,margin=1in]{geometry}
\usepackage{amsmath,amssymb,amsthm,mathtools}
\usepackage{mathrsfs}
\usepackage{enumitem}
\usepackage{hyperref}

\newtheorem{theorem}{Theorem}[section]
\newtheorem{proposition}[theorem]{Proposition}
\newtheorem{lemma}[theorem]{Lemma}
\newtheorem{corollary}[theorem]{Corollary}
\newtheorem{remark}[theorem]{Remark}
\numberwithin{equation}{section}

\newcommand{\T}{\mathbb T}
\newcommand{\R}{\mathbb R}
\newcommand{\C}{\mathbb C}
\newcommand{\Lcal}{\mathcal L}
\newcommand{\ChatS}{\widehat C_{\rm S}}
\newcommand{\one}{\mathbf 1}
\newcommand{\norm}[1]{\left\lVert #1\right\rVert}
\newcommand{\ip}[2]{\left\langle #1,#2\right\rangle}

\title[Optimal Stabilization Prefactors for Schr\"odinger]{Asymptotic Growth of Optimal Stabilization Prefactors for the Schr\"odinger Equations}
\author{Shanlin Huang, Changqin Quan, Gengsheng Wang}
\date{}

\address {Shanlin Huang, School of Mathematics (Zhuhai), Sun Yat-sen University, Zhuhai 519082, Guangdong, China}
\email{huangshlin6@mail.sysu.edu.cn}

\address {Changqin Quan, Graduate School of System Informatics, Kobe University, Kobe, Japan}
\email{quanchqin@gmail.com}

\address {Gengsheng Wang, Hetao Institute of Mathematics and Interdisciplinary Studies (Shenzhen),  518017, China}
\email{wanggs62@yeah.net}

\begin{document}
\maketitle
\begin{abstract}
We study the asymptotic behavior of the optimal stabilization prefactor for
the free Schr\"odinger equation, optimized over all bounded stationary
linear feedbacks. For a prescribed decay rate $\delta>0$, let
$\ChatS(\delta)$ denote this optimal prefactor. Assuming that the uncontrolled
region contains a nonempty open set and that the observability cost satisfies
$C(T)\le C_0e^{C_0/T}$ for small $T>0$, we prove
$$
e^{c\sqrt{\delta}}\le \ChatS(\delta)\le e^{C\sqrt{\delta}},
\qquad \delta\gg1.
$$
The lower bound is based on the uncontrolled open hole and a quantitative
cutoff construction adapted to the Laplacian, while the upper bound combines
an exponentially weighted Gramian with the small-time observability cost
$C(T)\lesssim e^{C/T}$.

The results apply to controlled Schr\"odinger equations in three settings:
flat tori, including a strip-control example beyond the geometric control
condition; the whole space $\mathbb{R}^n$ with exterior-ball control; and
bounded smooth domains satisfying the geometric control condition for
generalized geodesics.

\end{abstract}

\section{Introduction and main result}
\label{sec:introduction}

Let $\Omega$ be one of the following:
\begin{enumerate}[label=\rm(\roman*)]
\item $\Omega=\R^n$;
\item $\Omega=\T^n:=(\R/(2\pi\mathbb Z))^n$;
\item a bounded connected domain $\Omega\subset\R^n$ with smooth boundary.     
\end{enumerate}
We consider the following internally controlled free Schr\"odinger equation over $\Omega$:
\begin{equation}\label{eq:schrodinger}
\begin{cases}
\partial_t y(t,x)
=
i\Delta_\Omega y(t,x)
+\one_\omega(x)u(t,x),
& (t,x)\in(0,\infty)\times\Omega,\\
y(0,x)=y_0(x),
& x\in\Omega.
\end{cases}
\end{equation}
Here $\omega\subset\Omega$ is a nonempty open control region, $\one_\omega$ denotes the characteristic function of $\omega$, 
  $\Delta_\Omega$ is the usual Laplacian in cases {\rm(i)}--{\rm(ii)}
and the Dirichlet Laplacian in case {\rm(iii)}. Thus
\begin{equation}\label{eq:laplacian-domain}
D(\Delta_\Omega)=
\begin{cases}
H^2(\Omega),&\Omega=\R^n\text{ or }\T^n,\\
H^2(\Omega)\cap H_0^1(\Omega),
&\Omega\subset\R^n\text{ is bounded and smooth}.
\end{cases}
\end{equation}

Let 
\begin{equation}\label{eq:free-operators}
X=L^2(\Omega),\qquad A:=i\Delta_\Omega,
\qquad D(A):=D(\Delta_\Omega),
\qquad B:=\one_\omega.
\end{equation}
Then we have
\begin{equation}\label{eq:free-operator-properties}
A^*=-A,
\qquad B^*=B=B^2,
\qquad \norm{B}_{\Lcal(X)}=1.
\end{equation}
Hence $A$ generates the unitary group $U(t):=e^{tA}=e^{it\Delta_\Omega}$,
and
\begin{equation}\label{eq:unitary-properties}
U(t)^*=U(-t),
\qquad \norm{U(t)}_{\Lcal(X)}=1,
\qquad t\in\R.
\end{equation}
We take the complex inner product to be linear in the first variable.

Stabilization of the equation~\eqref{eq:schrodinger} is an important
problem in control theory and gives rise to many significant questions.
Among them, two fundamental ones are: under what conditions a prescribed
exponential decay rate can be achieved, and how a feedback law achieving
such a decay rate can be constructed; see, for instance,
\cite{DusserRabah2000,Vest2013,WangXu2017,MaWangYu2023,
BhandariCapistranoFilhoMajumdarTanaka2024,
FragnelliMoumniSalhi2024,Nguyen2024}
and the references therein.

The problem studied in this paper is of a different nature.
We are concerned with the optimal stabilization prefactor in the
exponential stability estimate as the prescribed decay rate becomes large.
More precisely, for $K\in\Lcal(X)$, define the closed-loop generator
\begin{equation}\label{eq:closed-loop-generator}
A_K:=A+BK,
\qquad D(A_K):=D(A).
\end{equation}
For $\delta>0$, we introduce the stabilization prefactor
\begin{equation}\label{eq:CSdeltaK}
C_{\rm S}(\delta,K)
:=
\sup_{t\ge0}e^{\delta t}
\norm{e^{tA_K}}_{\Lcal(X)}
\in[1,+\infty],
\end{equation}
and its optimal value
\begin{equation}\label{eq:ChatS}
\ChatS(\delta)
:=
\inf_{K\in\Lcal(X)}C_{\rm S}(\delta,K),
\qquad \delta>0.
\end{equation}
Thus $\ChatS(\delta)$ is the optimal stabilization prefactor over all
bounded stationary feedbacks for the prescribed decay rate $\delta$, and
our problem is to determine its asymptotic growth as $\delta\to\infty$.

To the best of our knowledge, the asymptotic behavior of the optimal
stabilization prefactor $\ChatS(\delta)$ was first studied by Quan and
Wang~\cite{QuanWang2026} for finite-dimensional linear systems, where a
polynomial growth law is obtained through controllability-chain arguments.
The present paper is self-contained and treats the free Schr\"odinger
equation. Although its lower-bound argument is motivated by the
finite-dimensional chain idea, the mechanism is different: it uses localized
cutoff functions in an uncontrolled open region and the locality of the
Laplacian, leading instead to the exponential-square-root growth law.

To study this problem, we make the following assumptions.

\medskip
\noindent\textbf{(H1) Local uncontrolled region.}
There exist nonempty open sets $O_0,O$ such that
\begin{equation}\label{eq:uncontrolled-open-set}
O_0\Subset O\Subset\Omega\setminus\overline\omega.
\end{equation}
In the torus case, $O_0$ and $O$ are chosen inside a single flat coordinate chart.

\medskip
\noindent\textbf{(H2) Quantitative small-time observability.}
For every $T>0$, there exists an observability constant $C(T)>0$ such that
\begin{equation}\label{eq:observability-inequality}
\norm{z}_X^2
\le
C(T)\int_0^T\norm{BU(t)z}_X^2\,dt,
\qquad z\in X.
\end{equation}
Moreover, there exist constants $C_0,T_0>0$ such that the constants $C(T)$
can be chosen to satisfy
\begin{equation}\label{eq:observability-growth}
C(T)\le C_0e^{C_0/T},
\qquad 0<T\le T_0.
\end{equation}
Only the small-time estimate \eqref{eq:observability-growth} is used in the
large-$\delta$ asymptotic analysis. 

\begin{remark}\label{remark1.1-9-16}
$(i)$ Assumption {\rm(H1)} is equivalent to requiring that
$\Omega\setminus\overline\omega$ contain a nonempty open set. Indeed, if
$W\subset\Omega\setminus\overline\omega$ is nonempty and open, then one
can choose nonempty open sets
\[
O_0\Subset O\Subset W.
\]
In the torus case, these sets can be chosen inside a single flat coordinate chart.
 The nested pair $O_0\Subset O$ is introduced only for convenience
in the cutoff construction used later in the proof of the lower bound.

$(ii)$ Assumption {\rm(H1)} means that the control region $\omega$ leaves a
genuine open hole in the interior of $\Omega$.

$(iii)$ Assumption {\rm(H2)} is satisfied in several important situations,
which are presented in Section~\ref{sec:consequences}.
\end{remark}

Our main result identifies the optimal exponential growth scale of the
stabilization prefactor.

\begin{theorem}\label{thm:main}
Assume {\rm(H1)} and {\rm(H2)}. Then there exist constants
$c,C,\delta_0>0$ such that
\begin{equation}\label{eq:main-two-sided}
e^{c\sqrt\delta}
\le
\ChatS(\delta)
\le
e^{C\sqrt\delta},
\qquad \delta\ge\delta_0.
\end{equation}
Moreover,
\begin{equation}\label{eq:loglog-limit}
\lim_{\delta\to+\infty}
\frac{\log\log\ChatS(\delta)}{\log\delta}
=
\frac12.
\end{equation}
In particular,
\[
\log\ChatS(\delta)\asymp\sqrt\delta
\qquad\text{as }\delta\to+\infty .
\]
\end{theorem}

The following result gives a complementary observation to
Theorem~\ref{thm:main} by describing the extreme case in which the whole
domain is controlled.
\begin{theorem}\label{theorem1.3-9-18}
If $|\Omega\setminus\omega|=0$, equivalently $B=I$ as an operator on $X$,
then
\begin{equation}\label{eq:full-control-prefactor}
\ChatS(\delta;\omega)=1,
\qquad \delta>0.
\end{equation}
\end{theorem}

\begin{remark}\label{rem:main}
Theorem~\ref{thm:main} has several points worth emphasizing.

\begin{enumerate}[label=\rm(\roman*)]

\item
The theorem extends the optimal-prefactor problem studied by Quan and
Wang~\cite{QuanWang2026} from finite-dimensional systems to the free
Schr\"odinger equation. Their finite-dimensional argument relies on
Brunovsk\'y controllability chains and cannot be directly applied in the
present infinite-dimensional setting.

\item
Assumption {\rm(H1)} is used for the lower bound in
\eqref{eq:main-two-sided}. The key idea is to exploit a fundamental property of the Laplacian,
in the quantitative form given by Lemma~\ref{lem:chain-cutoff}, to
construct chains of arbitrary finite length inside an uncontrolled open
region away from the control set $\omega$.
 By the locality of the
Laplacian, all successive elements of the chain remain supported in the same
region, and hence the feedback term vanishes along the chain. This construction is inspired by the finite-dimensional Brunovsk\'y chain
argument in \cite{QuanWang2026}, but its realization here is entirely
different and is based on the existence of an uncontrolled open set. 

\item
Assumption {\rm(H2)} is used for the upper bound. The key idea is to establish
a direct quantitative relation between the small-time observability constant
and the optimal stabilization prefactor. More precisely, an estimate of the
form
$C(T)\le Ce^{C/T}$ as $T\to0$, together with the choice
$T\asymp\delta^{-1/2}$, yields
$\ChatS(\delta)\le e^{C\sqrt{\delta}}$ as $\delta\to\infty$.
Combined with the independent lower bound, this determines the optimal
exponential-square-root growth scale in \eqref{eq:main-two-sided}.
This relation is further quantified in
Corollary~\ref{cor:optimal-observability-exponent}.

\item
The full-control case provides a contrasting asymptotic regime.
If the uncontrolled region contains a nonempty open set, then
Theorem~\ref{thm:main} gives the exponential-square-root growth of the
optimal stabilization prefactor. If
$|\Omega\setminus\omega|=0$, then
\[
\ChatS(\delta;\omega)=1 .
\]
The intermediate case where $\Omega\setminus\omega$ has positive measure
but contains no nonempty open set is not covered by the present results
and remains open. Moreover, the constants in Theorem~\ref{thm:main}
need not remain uniform when the uncontrolled open region shrinks to zero,
unless the uniform assumptions of Theorem~\ref{thm:omega-robustness}
are satisfied.

\end{enumerate}
\end{remark}

The next theorem describes the robustness of the optimal stabilization
prefactor with respect to variations of the control region. When the
dependence on $\omega$ is made explicit, we write $\ChatS(\delta;\omega)$
for the corresponding optimal stabilization prefactor.

For a family of control regions, we use the following uniform versions of
{\rm(H1)} and {\rm(H2)}.

\medskip
\noindent\textbf{($\widehat{\rm H1}$) Uniform local uncontrolled region.}
There exist fixed nonempty open sets
$O_0\Subset O\Subset\Omega$ such that
\begin{equation}\label{eq:uniform-hole}
O\Subset\Omega\setminus\overline{\omega_\eta},
\qquad \eta\in\mathcal I .
\end{equation}
In the torus case, $O_0$ and $O$ are chosen 
inside a fixed flat coordinate chart.

\medskip
\noindent\textbf{($\widehat{\rm H2}$) Uniform quantitative small-time observability.}
There exist constants $C_0,T_0>0$, independent of $\eta$, such that
\begin{equation}\label{eq:uniform-observability}
\norm{z}_X^2
\le
C_0e^{C_0/T}
\int_0^T
\norm{\one_{\omega_\eta}U(t)z}_X^2\,dt,
\qquad z\in X,
\end{equation}
for every $\eta\in\mathcal I$ and every $0<T\le T_0$.

\begin{theorem}\label{thm:omega-robustness}
Let
$\{\omega_\eta\}_{\eta\in\mathcal I}$ be a family of nonempty open
control regions satisfying the uniform assumptions
$(\widehat{\rm H1})$ and $(\widehat{\rm H2})$.
Then there exist $c,C,\delta\sb{0}>0$, independent of $\eta$, such that
\begin{equation}\label{eq:uniform-prefactor}
e^{c\sqrt\delta}
\le
\ChatS(\delta;\omega\sb{\eta})
\le
e^{C\sqrt\delta},
\qquad \delta\ge\delta\sb{0},
\quad \eta\in\mathcal I.
\end{equation}
\end{theorem}

\begin{remark}\label{remark1.5-9-16}
$(i)$
Theorem~\ref{thm:omega-robustness} shows that the
exponential-square-root growth law is uniform as long as a fixed
uncontrolled open region and a uniform small-time observability estimate
are preserved. In particular, enlarging the control region preserves the
observability estimate, provided that a fixed open hole remains, while
shrinking the control region preserves the lower-bound mechanism and the
upper bound remains uniform whenever the observability estimate is uniform.

$(ii)$
A simple sufficient condition for
$(\widehat{\rm H1})$ and $(\widehat{\rm H2})$ is the following. Suppose
that there exist two fixed open sets $\omega_-$ and $\omega_+$ such that
$$
\omega_-\subset\omega_\eta\subset\omega_+,
\qquad \eta\in\mathcal I.
$$
Assume that there exist fixed nonempty open sets
$$
O_0\Subset O\Subset
\Omega\setminus\overline{\omega_+}.
$$
In the torus case, $O_0$ and $O$ are chosen  inside a fixed flat coordinate chart.
 Since $\omega_\eta\subset\omega_+$,
$$
O_0\Subset O\Subset
\Omega\setminus\overline{\omega_\eta},
\qquad \eta\in\mathcal I,
$$
and hence $(\widehat{\rm H1})$ holds.

Assume further that $\omega_-$ satisfies {\rm(H2)} with constants
$C_0,T_0>0$. Since $\omega_-\subset\omega_\eta$,
$$
\int_0^T
\norm{\one_{\omega_-}U(t)z}_X^2\,dt
\le
\int_0^T
\norm{\one_{\omega_\eta}U(t)z}_X^2\,dt,
$$
and therefore
$$
\norm{z}_X^2
\le
C_0e^{C_0/T}
\int_0^T
\norm{\one_{\omega_\eta}U(t)z}_X^2\,dt,
\qquad
z\in X,\quad 0<T\le T_0,
$$
for every $\eta\in\mathcal I$. Thus $(\widehat{\rm H2})$ holds with
constants independent of $\eta$.
\end{remark}

\section{Proofs of the main theorems}

This section is divided into three subsections. The first establishes the
lower bound in Theorem~\ref{thm:main}, and the second proves the upper
bound. The final subsection completes the proofs of
Theorems~\ref{thm:main}, \ref{theorem1.3-9-18}, and
\ref{thm:omega-robustness}, and derives
Corollary~\ref{cor:optimal-observability-exponent}.

\subsection{Quantitative cutoffs and the lower bound}
\label{sec:lower}

Our strategy is as follows. The lower bound must hold uniformly with
respect to the feedback operator $K$. Instead of estimating the feedback
term, we use the hole condition (H1) to make it disappear on suitable
test functions. The key is to work with the adjoint generator
$A_K^*=-i\Delta_\Omega+K^*B$: if a smooth function is supported in the
uncontrolled open set $O$ in \eqref{eq:uncontrolled-open-set}, then
$B=\one_\omega$ annihilates it, and hence the feedback term $K^*B$
vanishes. Since the Laplacian is local, the same property holds for all
successive Laplacian iterates. Thus, the adjoint closed-loop dynamics
coincide with the free Schr\"odinger dynamics along a finite chain.

The remaining issue is to control the size of these iterates. For each
integer $r\ge1$, Lemma~\ref{lem:chain-cutoff} provides a function
$\phi_r\in C_c^\infty(O)$ satisfying
$\norm{\phi_r}_X\ge c_0$ and
$\norm{\Delta_\Omega^r\phi_r}_X\le(C_1r)^{2r}$, with constants independent
of $r$ and $K$. Combining
\[
(A_K^*)^r\phi_r=(-i)^r\Delta_\Omega^r\phi_r
\]
with the negative-power estimate for $(A_K^*)^{-r}$ gives
\[
C_{\rm S}(\delta,K)
\ge c_0\left(\frac{\delta}{C_1^2r^2}\right)^r .
\]
Choosing $r\sim\sqrt\delta$ yields the lower bound
$C_{\rm S}(\delta,K)\ge e^{c\sqrt\delta}$ uniformly in $K$. Thus, the
hole condition removes the feedback from the argument, while the
quantitative cutoff estimate determines the $\sqrt\delta$ scale.

We begin with the negative-power estimate needed in this argument.
The following standard semigroup estimate follows from the
resolvent-power formula; see \cite[Corollary~II.1.11]{EngelNagel2000}.

\begin{lemma}\label{lem:negative-power}
Let $G$ generate a strongly continuous semigroup on a Hilbert space and
assume that, for some $M\ge1$ and $\delta>0$,
\[
\norm{e^{tG}}\le Me^{-\delta t},
\qquad t\ge0.
\]
Then $G$ is invertible and, for every integer $r\ge1$,
\[
G^{-r}
=\frac{(-1)^r}{(r-1)!}
\int_0^\infty t^{r-1}e^{tG}\,dt.
\]
In particular,
\[
\norm{G^{-r}}\le M\delta^{-r}.
\]
\end{lemma}

We also use the following standard finite-order Ehrenpreis cutoff
construction; see \cite[Proposition~3.2.1]{Treves2022}.

\begin{lemma}\label{lem:cutoff}
Let $O_0\Subset O\Subset\Omega$ be nonempty open sets, with $O$ contained in a flat coordinate chart in the torus case.
 There exists a constant
$C_0\ge1$ such that, for every integer $N\ge1$, there exists
$\chi_N\in C_c^\infty(O)$ satisfying
\[
0\le\chi_N\le1,
\qquad \chi_N=1\ \text{on }O_0,
\]
and
\[
\norm{\partial^\alpha\chi_N}_{L^\infty(\Omega)}
\le C_0(C_0N)^{|\alpha|},
\qquad |\alpha|\le N.
\]
\end{lemma}

We next state the quantitative cutoff estimate for the lower bound.

\begin{lemma}[Quantitative Schr\"odinger chain cutoff]\label{lem:chain-cutoff}
Assume \eqref{eq:uncontrolled-open-set}. Then there exist constants
$c_0,C_1>0$ such that, for every integer $r\ge1$, there exists
$\phi_r\in C_c^\infty(O)$ satisfying
\begin{equation}\label{eq:phi-lower}
\norm{\phi_r}_{L^2(\Omega)}\ge c_0
\end{equation}
and
\begin{equation}\label{eq:laplacian-upper}
\norm{\Delta_\Omega^r\phi_r}_{L^2(\Omega)}
\le(C_1r)^{2r}.
\end{equation}
\end{lemma}
\begin{proof}
We first show \eqref{eq:phi-lower}. Indeed, by Lemma~\ref{lem:cutoff},
applied to the sets $O_0\Subset O$ in
\eqref{eq:uncontrolled-open-set} with $N=2r$, there exists a cutoff
function
\begin{equation}\label{eq:phi-cutoff-definition}
\phi_r:=\chi_{2r}.
\end{equation}
Since $\phi_r=1$ on $O_0$, we have
\[
\norm{\phi_r}_{L^2(\Omega)}^2
\ge
\int_{O_0}|\phi_r(x)|^2\,dx
=
|O_0|.
\]
Therefore, \eqref{eq:phi-lower} holds with
$c_0:=|O_0|^{1/2}>0$.

We next show \eqref{eq:laplacian-upper}. To this end, we estimate
$\Delta_\Omega^r\phi_r$. Three observations are as follows.

First, by \eqref{eq:phi-cutoff-definition} and the properties of the
cutoff function in Lemma~\ref{lem:cutoff}, we have
$\phi_r\in C_c^\infty(O)$ with $O\Subset\Omega$. In the whole-space and
bounded-domain cases, $\Delta_\Omega$ coincides with the Euclidean
Laplacian on $O$; in the torus case, the same holds in the flat coordinate
chart containing $O$. Hence
\[
\Delta_\Omega
=
\sum_{j=1}^n\partial_{x_j}^2
\]
on $O$. Therefore, the multinomial formula gives
\begin{equation}\label{eq:laplacian-multinomial}
\Delta_\Omega^r\phi_r
=
\left(\sum_{j=1}^n\partial_{x_j}^2\right)^r\phi_r
=
\sum_{\substack{\beta\in\mathbb N^n\\|\beta|=r}}
\frac{r!}{\beta!}\,\partial^{2\beta}\phi_r,
\end{equation}
where $\beta!:=\beta_1!\cdots\beta_n!$.

Second, by Lemma~\ref{lem:cutoff} with $N=2r$ and
\eqref{eq:phi-cutoff-definition}, for every
$\beta$ with $|\beta|=r$, we have
\begin{align}
\norm{\partial^{2\beta}\phi_r}_{L^2(\Omega)}
&\le
|O|^{1/2}
\norm{\partial^{2\beta}\phi_r}_{L^\infty(\Omega)}
\notag\\
&\le
|O|^{1/2}C_0(2C_0r)^{2r}.
\label{eq:multi-derivative-bound}
\end{align}

Third, we have
\begin{equation}\label{eq:multinomial-coefficient-sum}
\sum_{\substack{\beta\in\mathbb N^n\\|\beta|=r}}
\frac{r!}{\beta!}
=
n^r.
\end{equation}

Combining \eqref{eq:laplacian-multinomial},
\eqref{eq:multi-derivative-bound}, and
\eqref{eq:multinomial-coefficient-sum}, we obtain
\begin{equation}\label{eq:laplacian-cutoff-computation}
\norm{\Delta_\Omega^r\phi_r}_{L^2(\Omega)}
\le
n^r|O|^{1/2}C_0(2C_0r)^{2r}.
\end{equation}
After increasing a constant $C_1\ge1$, depending only on $n$, $O_0$, and
$O$, the right-hand side of \eqref{eq:laplacian-cutoff-computation} is
bounded by $(C_1r)^{2r}$ for every $r\ge1$. This proves
\eqref{eq:laplacian-upper}, and completes the proof.
\end{proof}

We now establish the lower bound in \eqref{eq:main-two-sided}.

\begin{proposition}[Feedback-independent lower bound]\label{prop:lower}
Assume \eqref{eq:uncontrolled-open-set}. Then there exist constants $c>0$ and
$\delta_1>0$ such that, for every $K\in\Lcal(X)$ and every
$\delta\ge\delta_1$,
\begin{equation}\label{eq:lower-final}
C_{\rm S}(\delta,K)\ge e^{c\sqrt\delta}.
\end{equation}
In particular,
\begin{equation}\label{eq:ChatS-lower}
\ChatS(\delta)\ge e^{c\sqrt\delta},
\qquad \delta\ge\delta_1.
\end{equation}
\end{proposition}

\begin{proof}
Recall from  \eqref{eq:free-operators} that
\[
X=L^2(\Omega;\C),\qquad
A=i\Delta_\Omega,\qquad
D(A)=D(\Delta_\Omega),\qquad
B=\one_\omega.
\]
Fix $K\in\Lcal(X)$. If
$C_{\rm S}(\delta,K)=+\infty$, then \eqref{eq:lower-final} is immediate.
We therefore assume that
\begin{equation}\label{eq:CS-finite}
C_{\rm S}(\delta,K)<+\infty.
\end{equation}

It follows from \eqref{eq:free-operator-properties} and
\eqref{eq:closed-loop-generator} that
\begin{equation}\label{eq:adjoint-closed-loop}
A_K^*=-i\Delta_\Omega+K^*B,
\qquad D(A_K^*)=D(A).
\end{equation}
We will use this identity to show that, on the chain
$\{\Delta_\Omega^j\phi_r\}_{j\ge0}$, the feedback term in the adjoint
closed-loop operator vanishes, so that it acts as the free Schrödinger
operator.

Fix an integer $r\ge1$ and let $\phi_r$ be given by
Lemma~\ref{lem:chain-cutoff}. We first claim
\begin{equation}\label{eq:feedback-invariant-chain}
(A_K^*)^r\phi_r=(-i)^r\Delta_\Omega^r\phi_r.
\end{equation}

Indeed, since the Laplacian is local and
$\operatorname{supp}\phi_r\Subset O$,
 we have
\begin{equation}\label{eq:chain-support}
\operatorname{supp}(\Delta_\Omega^j\phi_r)
\subset\operatorname{supp}\phi_r\Subset O,
\qquad j\ge0.
\end{equation}
Since \eqref{eq:uncontrolled-open-set} gives
$\overline O\cap\omega=\emptyset$, it follows from \eqref{eq:chain-support} and
$B=\one_\omega$ that
\begin{equation}\label{eq:B-chain-zero}
B\Delta_\Omega^j\phi_r=0,
\qquad j=0,1,\ldots,r-1.
\end{equation}
Meanwhile, by \eqref{eq:chain-support} and
\eqref{eq:adjoint-closed-loop}, we have
\begin{equation}\label{eq:chain-domain}
\Delta_\Omega^j\phi_r\in C_c^\infty(O)\subset D(A_K^*),
\qquad j\ge0.
\end{equation}
Combining
\eqref{eq:adjoint-closed-loop}, \eqref{eq:B-chain-zero}, and
\eqref{eq:chain-domain}, we obtain
\begin{equation}\label{eq:one-step-chain}
A_K^*\Delta_\Omega^j\phi_r
=-i\Delta_\Omega^{j+1}\phi_r,
\qquad j=0,1,\ldots,r-1.
\end{equation}
Applying \eqref{eq:one-step-chain} successively gives  \eqref{eq:feedback-invariant-chain}.

We next claim 
\begin{equation}\label{eq:phi-negative-power}
\norm{\phi_r}_{L^2(\Omega)}
\le C_{\rm S}(\delta,K)\delta^{-r}
\norm{\Delta_\Omega^r\phi_r}_{L^2(\Omega)}.
\end{equation}

In fact,  it follows from \eqref{eq:CSdeltaK} and \eqref{eq:CS-finite} that
\[
\norm{e^{tA_K}}_{\Lcal(X)}
\le C_{\rm S}(\delta,K)e^{-\delta t},
\qquad t\ge0.
\]
This, along with 
\[
e^{tA_K^*}=\bigl(e^{tA_K}\bigr)^*,
\qquad t\ge0,
\]
 yields
\begin{equation}\label{eq:adjoint-decay}
\norm{e^{tA_K^*}}_{\Lcal(X)}
\le C_{\rm S}(\delta,K)e^{-\delta t},
\qquad t\ge0.
\end{equation}
Applying Lemma~\ref{lem:negative-power} to $G=A_K^*$ with
$M=C_{\rm S}(\delta,K)$ and using \eqref{eq:adjoint-decay}, we obtain
\begin{equation}\label{eq:adjoint-negative-power-bound}
\norm{(A_K^*)^{-r}}_{\Lcal(X)}
\le C_{\rm S}(\delta,K)\delta^{-r}.
\end{equation}
Meanwhile, applying $(A_K^*)^{-r}$ to \eqref{eq:feedback-invariant-chain} gives
\begin{equation}\label{eq:phi-negative-power-identity}
\phi_r=(-i)^r(A_K^*)^{-r}\Delta_\Omega^r\phi_r.
\end{equation}
Taking norms in \eqref{eq:phi-negative-power-identity} and using
\eqref{eq:adjoint-negative-power-bound}, we obtain \eqref{eq:phi-negative-power}.

Now \eqref{eq:phi-lower}, \eqref{eq:laplacian-upper}, and
\eqref{eq:phi-negative-power} imply
\begin{equation}\label{eq:pre-optimize-r}
C_{\rm S}(\delta,K)
\ge c_0\left(\frac{\delta}{C_1^2r^2}\right)^r.
\end{equation}
Choose $\varepsilon>0$ such that
\begin{equation}\label{eq:epsilon-choice}
C_1^2\varepsilon^2<\frac14,
\end{equation}
and set
\begin{equation}
r:=\lfloor\varepsilon\sqrt\delta\rfloor.
\end{equation}
If $\varepsilon\sqrt\delta\ge2$, then
\begin{equation}\label{eq:r-bounds}
\frac{\varepsilon}{2}\sqrt\delta
\le r\le\varepsilon\sqrt\delta.
\end{equation}
The upper bound in \eqref{eq:r-bounds} and \eqref{eq:epsilon-choice} give
\begin{equation}\label{eq:ratio-lower-four}
\frac{\delta}{C_1^2r^2}
\ge\frac1{C_1^2\varepsilon^2}>4.
\end{equation}
Combining \eqref{eq:pre-optimize-r}, \eqref{eq:ratio-lower-four}, and the
lower bound in \eqref{eq:r-bounds}, we obtain
\begin{equation}
C_{\rm S}(\delta,K)
\ge c_0\exp\!\left(
\frac{\varepsilon\log4}{2}\sqrt\delta
\right).
\end{equation}
Since $c_0>0$ is independent of $K$ and $\delta$, after increasing the lower
threshold we can absorb $c_0$ into the exponential. Thus there exist
$c>0$ and $\delta_1>0$, independent of $K$, such that
\eqref{eq:lower-final} holds for $\delta\ge\delta_1$. Taking the infimum over
$K\in\Lcal(X)$ in \eqref{eq:lower-final} and using \eqref{eq:ChatS} gives
\eqref{eq:ChatS-lower}.
\end{proof}

\begin{remark}\label{remk-2.5}
The idea of constructing a feedback-invariant chain is inspired by the
finite-dimensional argument of Quan and Wang~\cite[Proposition~1]{QuanWang2026}.
In their setting, such a chain is generated by the Brunovsk\'y canonical
structure of the controllable pair $(A,B)$. In contrast, for the
Schr\"odinger equation considered here, there is no intrinsic
finite-dimensional controllability chain. The role of the chain is instead
played by functions localized in the uncontrolled region, whose construction
relies on the hole condition (H1) and the quantitative cutoff estimate in
Lemma~\ref{lem:chain-cutoff}. Thus, the underlying idea is related to the
finite-dimensional case, but the mechanism is genuinely different and
geometric in nature.
\end{remark}

\subsection{The extended Gramian and the upper bound}
\label{sec:upper}

The purpose of this subsection is to prove the upper bound in
Theorem~\ref{thm:main}. In contrast with the lower-bound argument, where
the uncontrolled region prevents any feedback from improving the
stabilization cost, the upper bound is obtained by constructing a suitable
feedback operator. The key ingredient is the extended Gramian associated
with the Schr\"odinger group. Under the quantitative short-time
observability assumption {\rm(H2)}, the observability estimate provides
quantitative coercivity of this Gramian, which allows us to construct a
bounded feedback operator and derive an exponential decay estimate with an
explicit stabilization prefactor.

\subsubsection{The extended Gramian in an abstract setting}
\label{sec:abstract-gramian}

The following two propositions are standard consequences of the
Lyapunov--Gramian theory and bounded similarity transformations. The second
provides the feedback construction and similarity estimate used in the
quantitative upper bound. The first is used in the proof of the second.
We state and prove both results in the present skew-adjoint setting for
completeness.

\begin{proposition}
\label{prop:extended-gramian}
Let $A$ be skew-adjoint on a Hilbert space $X$, let $Y$ be a Hilbert space,
let $B\in\Lcal(Y,X)$, and let $U(t)=e^{tA}$. For $\lambda>0$, define
$N_\lambda\in\Lcal(X)$ by
\begin{equation}\label{eq:extended-gramian}
N_\lambda x
:=\int_0^\infty e^{-\lambda t}
U(-t)BB^*U(t)x\,dt,
\qquad x\in X,
\end{equation}
where the integral is understood as an $X$-valued Bochner integral.
Then $N_\lambda$ is bounded, self-adjoint, and nonnegative. Moreover,
\begin{equation}\label{2.28-9-19}
N_\lambda D(A)\subset D(A),
\end{equation}
and
\begin{equation}\label{eq:lyapunov-identity}
AN_\lambda x-N_\lambda Ax+\lambda N_\lambda x=BB^*x,
\qquad x\in D(A).
\end{equation}
\end{proposition}

\begin{proof}
Fix $\lambda>0$.
We first prove that $N_\lambda$ is well defined as an element of
$\Lcal(X)$. Indeed, for $R>0$, define the truncated Gramian
\[
N_{\lambda,R}x
:=
\int_0^R e^{-\lambda t}U(-t)BB^*U(t)x\,dt,
\qquad x\in X .
\]
Since $U(t)$ is unitary, we have
\[
\bigl\|e^{-\lambda t}U(-t)BB^*U(t)\bigr\|_{\mathcal L(X)}
\le e^{-\lambda t}\|B\|^2,
\qquad t\ge0 .
\]
Therefore, for $R_2>R_1>0$,
\begin{equation}\label{eq:truncated-gramian-Cauchy}
\|N_{\lambda,R_2}-N_{\lambda,R_1}\|_{\mathcal L(X)}
\le
\|B\|^2\int_{R_1}^{R_2}e^{-\lambda t}\,dt .
\end{equation}
Hence $N_{\lambda,R}$ converges in operator norm as
$R\to+\infty$, and its limit is the bounded operator $N_\lambda$
defined by \eqref{eq:extended-gramian}.

We next prove \eqref{2.28-9-19}. Indeed, by \eqref{eq:extended-gramian},
for every $x\in X$,
\begin{equation}\label{eq:gramian-quadratic}
\langle N_\lambda x,x\rangle
=
\int_0^\infty e^{-\lambda t}
\|B^*U(t)x\|_Y^2\,dt\ge0 .
\end{equation}
Meanwhile, the operator $U(-t)BB^*U(t)$ is self-adjoint for every
$t\ge0$. Thus $N_\lambda$ is self-adjoint. Together with
\eqref{eq:gramian-quadratic}, $N_\lambda$ is self-adjoint and nonnegative.

Fix $x,y\in D(A)$ and define
\begin{equation}\label{eq:g-definition}
g(t):=
\langle B^*U(t)x,B^*U(t)y\rangle_Y,\qquad t\ge0 .
\end{equation}
Since $x,y\in D(A)$, the group invariance
$U(t)D(A)=D(A)$ for $t\ge0$ and the identity
$AU(t)=U(t)A$ on $D(A)$ for $t\ge0$ imply that
\[
\frac{d}{dt}U(t)x=U(t)Ax,\qquad
\frac{d}{dt}U(t)y=U(t)Ay,\qquad t\ge0 .
\]
Using that $B^*$ is bounded, we obtain from \eqref{eq:g-definition} that
\begin{equation}\label{eq:g-derivative}
g'(t)
=
\langle B^*U(t)Ax,B^*U(t)y\rangle_Y
+
\langle B^*U(t)x,B^*U(t)Ay\rangle_Y,\qquad t\ge0 .
\end{equation}
Meanwhile, since the function $g$ is bounded, we have
\[
e^{-\lambda t}g(t)\to0
\qquad\text{as }t\to+\infty .
\]
Integrating $(e^{-\lambda t}g(t))'$ over $(0,\infty)$ gives
\begin{equation}\label{eq:weighted-g-integration}
-g(0)
=
\int_0^\infty e^{-\lambda t}
\bigl(g'(t)-\lambda g(t)\bigr)\,dt .
\end{equation}
Using \eqref{eq:g-definition} (with $t=0$),
\eqref{eq:g-derivative}, \eqref{eq:weighted-g-integration},
and \eqref{eq:extended-gramian}, we obtain
\begin{equation}\label{eq:weak-lyapunov}
\langle BB^*x,y\rangle_X
=
-\langle N_\lambda Ax,y\rangle_X
-\langle N_\lambda x,Ay\rangle_X
+\lambda\langle N_\lambda x,y\rangle_X .
\end{equation}
Since $x\in D(A)$, the above identity implies
\[
\langle N_\lambda x,Ay\rangle_X
=
-\langle BB^*x,y\rangle_X
-\langle N_\lambda Ax,y\rangle_X
+\lambda\langle N_\lambda x,y\rangle_X,
\qquad y\in D(A).
\]
The right-hand side can be written as $\langle z,y\rangle_X$ for some
$z\in X$. Hence, by the characterization of the domain of the adjoint
operator,
\[
N_\lambda x\in D(A^*)=D(A).
\]
Since $x\in D(A)$ is arbitrary, we obtain \eqref{2.28-9-19}.

Finally, we prove \eqref{eq:lyapunov-identity}. Since $A^*=-A$,
\[
-\langle N_\lambda x,Ay\rangle_X
=
\langle AN_\lambda x,y\rangle_X .
\]
Substituting this relation into \eqref{eq:weak-lyapunov} yields
\[
\langle
AN_\lambda x-N_\lambda Ax+\lambda N_\lambda x-BB^*x,
y
\rangle_X=0,
\qquad y\in D(A).
\]
Since $D(A)$ is dense in $X$, we obtain
\eqref{eq:lyapunov-identity}. This completes the proof.
\end{proof}

\begin{proposition}
\label{prop:gramian-feedback}
Under the assumptions of Proposition~\ref{prop:extended-gramian},
let
$\lambda>0$ and
 assume
that $N_\lambda$ is boundedly invertible. Define
\[
K_\lambda:=-B^*N_\lambda^{-1}\in\Lcal(X,Y).
\]
Then the closed-loop generator
\begin{equation}\label{eq:abstract-closed-loop}
L_\lambda:=A+BK_\lambda
=A-BB^*N_\lambda^{-1},
\qquad D(L_\lambda)=D(A),
\end{equation}
satisfies
\begin{equation}\label{eq:similarity-semigroup}
e^{tL_\lambda}
=
e^{-\lambda t}N_\lambda U(t)N_\lambda^{-1},
\qquad t\ge0 .
\end{equation}
Moreover
\begin{equation}\label{eq:prefactor-condition-number}
\sup_{t\ge0}e^{\lambda t}
\norm{e^{tL_\lambda}}_{\Lcal(X)}
\le
\norm{N_\lambda}_{\Lcal(X)}
\norm{N_\lambda^{-1}}_{\Lcal(X)} .
\end{equation}
\end{proposition}

\begin{proof}
Fix $\lambda>0$.
Define
\[
\widetilde L_\lambda
:=
N_\lambda A N_\lambda^{-1}-\lambda I,
\qquad
D(\widetilde L_\lambda):=N_\lambda D(A).
\]
Since $N_\lambda$ is boundedly invertible, 
$N_\lambda A N_\lambda^{-1}$ is a bounded similarity transform of $A$.
Therefore it generates the group
$N_\lambda U(t)N_\lambda^{-1}$. Hence, by the shift
$-\lambda I$, the operator $\widetilde L_\lambda$ generates the group
\begin{equation}\label{eq:similar-group}
\widetilde S_\lambda(t)
=
e^{-\lambda t}N_\lambda U(t)N_\lambda^{-1},
\qquad t\in\mathbb R .
\end{equation}

Let $x=N_\lambda y$ with $y\in D(A)$. By the definition of
$\widetilde L_\lambda$ and the Lyapunov identity
\eqref{eq:lyapunov-identity}, we have
\begin{align}
\widetilde L_\lambda x
&=
N_\lambda Ay-\lambda N_\lambda y \notag\\
&=
AN_\lambda y-BB^*y \notag\\
&=
(A-BB^*N_\lambda^{-1})x .
\label{eq:similar-generator-action}
\end{align}
Since Proposition~\ref{prop:extended-gramian} gives
$N_\lambda D(A)\subset D(A)$, we have from \eqref{eq:abstract-closed-loop} and \eqref{eq:similar-generator-action} that
\[
D(\widetilde L_\lambda)=N_\lambda D(A)\subset D(A)
=D(L_\lambda),
\]
and therefore
\begin{equation}\label{eq:generator-extension}
\widetilde L_\lambda\subset L_\lambda .
\end{equation}

The operator $L_\lambda$ is a bounded perturbation of the skew-adjoint
operator $A$, and hence it is also a generator. Since
$\widetilde L_\lambda$ and $L_\lambda$ are generators, we may choose
$\mu>0$ sufficiently large such that
\[
\mu\in\rho(\widetilde L_\lambda)\cap\rho(L_\lambda).
\]
Let $x\in D(L_\lambda)$ and set
\[
h:=(\mu-L_\lambda)x .
\]
Since $\mu\in\rho(\widetilde L_\lambda)$, the operator
$\mu-\widetilde L_\lambda$ is surjective. Hence there exists
$z\in D(\widetilde L_\lambda)$ such that
\[
(\mu-\widetilde L_\lambda)z=h .
\]
Using \eqref{eq:generator-extension}, we obtain
\[
(\mu-L_\lambda)z
=
(\mu-L_\lambda)x .
\]
Since $\mu\in\rho(L_\lambda)$, the operator
$\mu-L_\lambda$ is injective. Therefore,
\[
z=x .
\]
Because $z\in D(\widetilde L_\lambda)$, we conclude that
\[
x\in D(\widetilde L_\lambda).
\]
Since $x\in D(L_\lambda)$ was arbitrary,
\[
D(L_\lambda)\subset D(\widetilde L_\lambda).
\]
Together with \eqref{eq:generator-extension}, this yields
\[
\widetilde L_\lambda=L_\lambda .
\]

The equality of the generators implies that their generated groups
coincide. Therefore, using \eqref{eq:similar-group}, we obtain
\begin{equation}\label{eq:similarity-semigroup}
e^{tL_\lambda}
=
e^{-\lambda t}N_\lambda U(t)N_\lambda^{-1},
\qquad t\ge0 .
\end{equation}

Finally, since $U(t)$ is unitary, \eqref{eq:similarity-semigroup} gives
\[
e^{\lambda t}\norm{e^{tL_\lambda}}_{\Lcal(X)}
=
\norm{N_\lambda U(t)N_\lambda^{-1}}_{\Lcal(X)}
\le
\norm{N_\lambda}_{\Lcal(X)}
\norm{N_\lambda^{-1}}_{\Lcal(X)},
\qquad t\ge0 .
\]
Taking the supremum over $t\ge0$ gives
\eqref{eq:prefactor-condition-number}. This completes the proof.
\end{proof}

\subsubsection{Quantitative upper bound for the Schr\"odinger equation}
\label{sec:schrodinger-upper}

We now return to the proof of the upper bound in
Theorem~\ref{thm:main}. To apply Proposition~\ref{prop:gramian-feedback}
to the Schr\"odinger system, we first record the following lemma.

\begin{lemma}
\label{lem:gramian-coercivity}
Assume \eqref{eq:observability-inequality}. Let $N_\lambda$ be defined by
\eqref{eq:extended-gramian} with $Y=X$ and $B=\one_\omega$. Then, for every
$\lambda>0$, there exists $c_\lambda>0$ such that
\begin{equation}\label{eq:gramian-coercivity}
N_\lambda\ge c_\lambda I .
\end{equation}
In particular, $N_\lambda$ is boundedly invertible.
\end{lemma}

\begin{proof}
Fix $\lambda>0$. By \eqref{eq:gramian-quadratic}
 and
\eqref{eq:free-operator-properties},
\[
\langle N_\lambda x,x\rangle
=
\int_0^\infty e^{-\lambda t}
\|BU(t)x\|_X^2\,dt,
\qquad x\in X.
\]
Fix $T>0$. Since $e^{-\lambda t}\ge e^{-\lambda T}$ for
$0\le t\le T$, it follows from the above identity and the observability
inequality \eqref{eq:observability-inequality} that
\[
\langle N_\lambda x,x\rangle
\ge
e^{-\lambda T}
\int_0^T\|BU(t)x\|_X^2\,dt
\ge
\frac{e^{-\lambda T}}{C(T)}\|x\|_X^2 .
\]
Hence \eqref{eq:gramian-coercivity} holds with
$c_\lambda=e^{-\lambda T}/C(T)>0$, and therefore $N_\lambda$ is boundedly
invertible. This completes the proof.
\end{proof}

We now use Lemma~\ref{lem:gramian-coercivity} and
Proposition~\ref{prop:gramian-feedback} to prove the following upper bound.

\begin{proposition}
\label{prop:upper}
Assume \eqref{eq:observability-inequality}--\eqref{eq:observability-growth}.
Then there exist constants $C_+>0$ and $\delta_2>0$ such that
\begin{equation}\label{eq:ChatS-upper}
\ChatS(\delta)\le e^{C_+\sqrt\delta},
\qquad \delta\ge\delta_2.
\end{equation}
\end{proposition}

\begin{proof}
Recall from  \eqref{eq:free-operators} that
\[
X=L^2(\Omega;\C),\qquad
A=i\Delta_\Omega,\qquad
D(A)=D(\Delta_\Omega),\qquad
B=\one_\omega.
\]
Let $C_0,T_0>0$ be the constants in
\eqref{eq:observability-growth}. Fix $\delta>0$ and let $N_\delta$ be the
operator \eqref{eq:extended-gramian} with $\lambda=\delta$ and $Y=X$.

By Proposition~\ref{prop:extended-gramian}, $N_\delta$ is self-adjoint
and nonnegative. Moreover, by \eqref{eq:gramian-quadratic},
\eqref{eq:free-operator-properties}, and \eqref{eq:unitary-properties},
we have
\[
\ip{N_\delta z}{z}
\le
\int_0^\infty e^{-\delta t}\norm{z}_X^2\,dt
=
\frac1\delta\norm{z}_X^2,\;\;z\in X.
\]
Hence
\begin{equation}\label{eq:N-upper}
\norm{N_\delta}\le\frac1\delta.
\end{equation}
Meanwhile, by Lemma~\ref{lem:gramian-coercivity}, $N_\delta$ is boundedly
invertible.

 We next estimate its inverse quantitatively. For every
$z\in X$, \eqref{eq:gramian-quadratic} and
\eqref{eq:free-operator-properties} give
\begin{equation}\label{eq:Ndelta-quadratic}
\ip{N_\delta z}{z}
=
\int_0^\infty e^{-\delta t}
\norm{BU(t)z}_X^2\,dt.
\end{equation}
Fix $T>0$. Since 
\[
e^{-\delta t}\ge e^{-\delta T},\;\;0\le t\le T,
\]
it follows from \eqref{eq:Ndelta-quadratic} that
\[
\ip{N_\delta z}{z}
\ge
e^{-\delta T}\int_0^T\norm{BU(t)z}_X^2\,dt,\;\;z\in X.
\]
Using the observability inequality
\eqref{eq:observability-inequality} in the above inequality, we obtain
\[
\ip{N_\delta z}{z}
\ge
\frac{e^{-\delta T}}{C(T)}\norm{z}_X^2,
\qquad z\in X.
\]
Since $N_\delta$ is self-adjoint, nonnegative, and boundedly invertible,
the above inequality yields
\begin{equation}\label{eq:N-inverse-upper}
\norm{N_\delta^{-1}}
\le
C(T)e^{\delta T}.
\end{equation}

Meanwhile, by Proposition~\ref{prop:gramian-feedback}, applied with
$\lambda=\delta$, and by \eqref{eq:free-operator-properties}, the feedback
\begin{equation}\label{eq:upper-feedback}
K_\delta
=
-B^*N_\delta^{-1}
=
-BN_\delta^{-1}
\end{equation}
belongs to $\Lcal(X)$. Therefore, by 
\eqref{eq:ChatS}, \eqref{eq:prefactor-condition-number},
\eqref{eq:N-upper},  \eqref{eq:N-inverse-upper}, and \eqref{eq:upper-feedback}, we have
\begin{equation}\label{eq:upper-master}
\ChatS(\delta)
\le
C_{\rm S}(\delta,K_\delta)
\le
\frac{C(T)}{\delta}e^{\delta T},
\qquad T>0.
\end{equation}
Furthermore, for $0<T\le T_0$, \eqref{eq:observability-growth} and
\eqref{eq:upper-master} give
\begin{equation}\label{eq:upper-before-optimization}
\ChatS(\delta)
\le
\frac{C_0}{\delta}
\exp\!\left(
\delta T+\frac{C_0}{T}
\right).
\end{equation}
Choose
\[
T:=\sqrt{\frac{C_0}{\delta}}.
\]
If
\[
\delta\ge \frac{C_0}{T_0^2},
\]
then $0<T\le T_0$, and hence
\eqref{eq:observability-growth} applies. For such $T$, we have
\begin{equation}\label{eq:balanced-exponent}
\delta T+\frac{C_0}{T}
=
2\sqrt{C_0\delta}.
\end{equation}
Thus, it follows from \eqref{eq:upper-before-optimization} and \eqref{eq:balanced-exponent} that
\begin{equation}\label{eq:upper-final-computation}
\ChatS(\delta)
\le
\frac{C_0}{\delta}
e^{2\sqrt{C_0\delta}}.
\end{equation}
After increasing the threshold $\delta_2$ and the constant $C_+$ if
necessary, \eqref{eq:upper-final-computation} yields
\eqref{eq:ChatS-upper}. This completes the proof.
\end{proof}

\subsection{Proofs of Theorems~\ref{thm:main},
\ref{theorem1.3-9-18}, and~\ref{thm:omega-robustness}}
\label{sec:schrodinger-upper}

We now give proofs of our main theorems.

\begin{proof}[Proof of Theorem~\ref{thm:main}]
Let $c,\delta_1>0$ be given by Proposition~\ref{prop:lower}, and let
$C_+,\delta_2>0$ be given by Proposition~\ref{prop:upper}. Set
\[
\delta_0:=\max\{\delta_1,\delta_2\}.
\]
Then \eqref{eq:ChatS-lower} and \eqref{eq:ChatS-upper} give
\[
e^{c\sqrt\delta}
\le
\ChatS(\delta)
\le
e^{C_+\sqrt\delta},
\qquad \delta\ge\delta_0,
\]
which proves \eqref{eq:main-two-sided}, with $C:=C_+$.

Taking logarithms in the above, we obtain
\[
c\sqrt\delta
\le
\log\ChatS(\delta)
\le
C_+\sqrt\delta,
\qquad \delta\ge\delta_0.
\]
Hence, for all sufficiently large $\delta$,
\[
\log c+\frac12\log\delta
\le
\log\log\ChatS(\delta)
\le
\log C_++\frac12\log\delta.
\]
Dividing by $\log\delta$ and letting $\delta\to+\infty$ gives
\eqref{eq:loglog-limit}. The preceding bounds on
$\log\ChatS(\delta)$ also give
\[
\log\ChatS(\delta)\asymp\sqrt\delta
\qquad\text{as }\delta\to+\infty.
\]
This completes the proof.
\end{proof}

\begin{proof}[Proof of Theorem~\ref{theorem1.3-9-18}]
Since $|\Omega\setminus\omega|=0$, we have $B=I$ on
$X$. For $\delta>0$, take
\begin{equation}\label{eq:full-control-feedback}
K\sb{\delta}:=-\delta I.
\end{equation}
Since $U(t)=e^{tA}$ is unitary,
\begin{equation}\label{eq:full-control-semigroup}
e^{t(A+BK\sb{\delta})}
=
e^{-\delta t}U(t),
\qquad t\ge0.
\end{equation}
Hence
\begin{equation}\label{eq:full-control-CS}
C_{\rm S}(\delta,K\sb{\delta})
=
\sup_{t\ge0}
e^{\delta t}
\norm{e^{t(A+BK\sb{\delta})}}_{\Lcal(X)}
=
1.
\end{equation}
Therefore $\ChatS(\delta;\omega)\le1$. On the other hand,
$C_{\rm S}(\delta,K)\ge1$ for every $K\in\Lcal(X)$, by taking $t=0$ in
\eqref{eq:CSdeltaK}. Thus $\ChatS(\delta;\omega)\ge1$, and
\eqref{eq:full-control-prefactor} follows. This completes the proof.
\end{proof}

\begin{proof}[Proof of Theorem~\ref{thm:omega-robustness}]
We first consider the lower bound. By $(\widehat{\rm H1})$, there exist
fixed nonempty open sets $O_0$ and $O$ such that
$$
O_0\Subset O\Subset
\Omega\setminus\overline{\omega_\eta},
\qquad \eta\in\mathcal I.
$$
Hence Lemma~\ref{lem:chain-cutoff} applies with the same sets $O_0$ and
$O$ for every $\eta\in\mathcal I$. The constants $c_0$ and $C_1$ in that
lemma can therefore be chosen independently of $\eta$. Inspecting the
proof of Proposition~\ref{prop:lower}, we see that the constants $c$ and
$\delta_1$ are determined by $c_0$ and $C_1$. Thus they can also be
chosen independently of $\eta$. Therefore,
$$
\ChatS(\delta;\omega_\eta)
\ge
e^{c\sqrt\delta},
\qquad
\delta\ge\delta_1,\quad \eta\in\mathcal I.
$$

We next consider the upper bound. By $(\widehat{\rm H2})$, there exist
constants $C_0,T_0>0$, independent of $\eta$, such that
$$
\norm{z}_X^2
\le
C_0e^{C_0/T}
\int_0^T
\norm{\one_{\omega_\eta}U(t)z}_X^2\,dt,
\qquad
z\in X,\quad 0<T\le T_0,
$$
for every $\eta\in\mathcal I$. For $T>T_0$, the estimate at time $T_0$
remains valid after enlarging the observation interval. Hence, for every
$\eta\in\mathcal I$, {\rm(H2)} holds with
$B=\one_{\omega_\eta}$, where the constants $C_0,T_0$ in
\eqref{eq:observability-growth} are independent of $\eta$.
Moreover, $B_\eta:=\one_{\omega_\eta}$ satisfies
$$
B_\eta^*=B_\eta=B_\eta^2,
\qquad
\norm{B_\eta}_{\Lcal(X)}=1,
\qquad \eta\in\mathcal I.
$$
Thus, it follows from the proof of Proposition~\ref{prop:upper} that the
constants $C_+$ and $\delta_2$ can be chosen independently of $\eta$.
Hence,
$$
\ChatS(\delta;\omega_\eta)
\le
e^{C_+\sqrt\delta},
\qquad
\delta\ge\delta_2,\quad \eta\in\mathcal I.
$$

Finally, set

$$
\delta_0:=\max\{\delta_1,\delta_2\}.
$$

Taking $C:=C_+$, the preceding lower and upper bounds give
\eqref{eq:uniform-prefactor}. This completes the proof.
\end{proof}

We conclude this subsection with a consequence of
\eqref{eq:upper-master} and Proposition~\ref{prop:lower}, which relates
the observability constants to the optimal stabilization prefactor and
yields a restriction on their small-time growth.

\begin{corollary}
\label{cor:optimal-observability-exponent}
Assume \eqref{eq:observability-inequality}. Then the following statements hold.

\begin{enumerate}[label={\rm(\roman*)}]
\item
For every $\delta>0$ and every $T>0$,
\begin{equation}\label{eq:obs-prefactor-relation}
\ChatS(\delta)
\le
\frac{C(T)}{\delta}e^{\delta T}.
\end{equation}

\item
Assume in addition {\rm(H1)}. Suppose that, for some
$C_0,T_0,\beta>0$, the observability constants $C(T)$ can be chosen so that
\[
C(T)\le
C_0\exp\!\left(\frac{C_0}{T^\beta}\right),
\qquad 0<T\le T_0.
\]
Then
\[
\beta\ge1.
\]
In particular, the observability constants cannot be chosen to satisfy
such an estimate with $0<\beta<1$.
\end{enumerate}
\end{corollary}

\begin{proof}
Part {\rm(i)} follows from the argument leading to
\eqref{eq:upper-master}, which uses only
\eqref{eq:observability-inequality}.

For part {\rm(ii)}, combining \eqref{eq:obs-prefactor-relation} with the
assumed bound on $C(T)$ gives
\[
\ChatS(\delta)
\le
\frac{C_0}{\delta}
\exp\!\left(
\delta T+\frac{C_0}{T^\beta}
\right),
\qquad 0<T\le T_0.
\]
For sufficiently large $\delta$, choose
\[
T=
\left(\frac{\beta C_0}{\delta}\right)^{1/(\beta+1)}.
\]
Then there exists $C_1>0$ such that
\[
\ChatS(\delta)
\le
\exp\!\left(
C_1\delta^{\beta/(\beta+1)}
\right)
\]
for all sufficiently large $\delta$.

On the other hand, {\rm(H1)} and Proposition~\ref{prop:lower} give
\[
\ChatS(\delta)\ge e^{c\sqrt\delta}
\]
for all sufficiently large $\delta$. Hence
\[
c\delta^{1/2}
\le
C_1\delta^{\beta/(\beta+1)}
\]
for all sufficiently large $\delta$. Therefore,
\[
\frac{\beta}{\beta+1}\ge\frac12,
\]
which is equivalent to $\beta\ge1$. This completes the proof.
\end{proof}

\section{Applications to Specific Geometries}
\label{sec:consequences}

In this section, we apply Theorem~\ref{thm:main} in three settings:
flat tori, the whole space $\R^n$, and bounded smooth domains.
For each setting, we verify {\rm(H1)} directly and use known quantitative
observability estimates to verify {\rm(H2)}. Theorem~\ref{thm:main} then
gives the corresponding asymptotic bounds for the optimal stabilization
prefactor.

\subsection{Flat tori}

Let $\Omega=\T^n$, and assume that the nonempty open set
$\omega\subset\T^n$ satisfies {\rm(H1)}, with $O_0$ and $O$  contained in a single flat coordinate chart.

Define the maximal avoiding-geodesic length by
\[
L_\omega
:=
\sup\left\lbrace
\ell>0:\ \text{there exists a unit-speed geodesic }
\gamma:[0,\ell]\to\T^n
\text{ such that }\gamma([0,\ell])\cap\omega=\emptyset
\right\rbrace.
\]
By \cite[Theorem~1.3]{Miller2004}, if
$L_\omega<+\infty$, then, after squaring the nonsquared observability
estimate in Miller and changing the constants, there exist $C_0,T_0>0$
such that
\begin{equation}\label{eq:torus-small-time-observability}
\norm{z}_{L^2(\T^n)}^2
\le
C_0e^{C_0/T}
\int_0^T
\norm{\one_\omega e^{it\Delta}z}_{L^2(\T^n)}^2\,dt,
\qquad
z\in L^2(\T^n),\quad 0<T\le T_0.
\end{equation}
For $T>T_0$, the estimate at time $T_0$ remains valid after enlarging
the observation interval. Hence \eqref{eq:observability-inequality} holds
for every $T>0$, and \eqref{eq:observability-growth} holds for
$0<T\le T_0$. Thus {\rm(H2)} is satisfied.

\begin{corollary}[Flat torus]\label{cor:torus}
Let $\Omega=\T^n$, and let $\omega\subset\T^n$ be a nonempty open set
satisfying {\rm(H1)}, with $O_0$ and $O$ contained in a flat coordinate
chart. If
\begin{equation}\label{eq:finite-Lomega-torus}
L_\omega<+\infty,
\end{equation}
then there exist $c,C,\delta_0>0$ such that
\begin{equation}\label{eq:torus-prefactor}
e^{c\sqrt\delta}
\le
\ChatS(\delta)
\le
e^{C\sqrt\delta},
\qquad \delta\ge\delta_0.
\end{equation}
In particular, \eqref{eq:loglog-limit} holds.
\end{corollary}

\begin{proof}
Assumption {\rm(H1)} is part of the hypotheses. By
\eqref{eq:finite-Lomega-torus} and
\cite[Theorem~1.3]{Miller2004}, the preceding argument verifies {\rm(H2)}.
The conclusion follows from Theorem~\ref{thm:main}. This completes the proof.
\end{proof}

The condition $L_\omega<+\infty$ is the geometric control condition on
the flat torus. In dimension one, it holds for every nonempty open
observation set. Hence Corollary~\ref{cor:torus} applies to every nonempty
open set $\omega\subset\T$ for which
$\T\setminus\overline\omega$ contains a nonempty open interval. Indeed,
the sets $O_0$ and $O$ in {\rm(H1)} can then be chosen inside this
interval and contained in a single flat coordinate chart.

The geometric control condition is not necessary for exact controllability
in rectangular geometries; see Jaffard~\cite{Jaffard1990} and
Komornik~\cite{Komornik1992}. For strip-shaped control regions in higher
dimensions, some geodesics avoid the control region for all time, whereas
the quantitative small-time observability estimate required here follows
from the one-dimensional estimate and Miller's product
result~\cite[Theorem~1.4]{Miller2004}. This gives the following corollary.

\begin{corollary}[Strip control on a flat torus]\label{cor:torus2}
Let $n\ge2$. There exists a nonempty open set
$\omega\subset\T^n$ satisfying {\rm(H1)} such that
\[
L_\omega=+\infty,
\]
while \eqref{eq:torus-prefactor} holds for some $c,C,\delta_0>0$.
\end{corollary}

\begin{proof}
Choose a nonempty open interval $I\subset\T$ such that
$\T\setminus\overline I$ contains a nonempty open interval, and set
$$
\omega:=I\times\T^{n-1}.
$$
Choose nonempty open intervals $J_0,J\subset\T$ such that
$$
J_0\Subset J\Subset\T\setminus\overline I,
$$
and choose nonempty open sets $V_0,V\subset\T^{n-1}$ such that
$$
V_0\Subset V
$$
and $J\times V$ is contained in a single flat coordinate chart of $\T^n$.
Set
$$
O_0:=J_0\times V_0,
\qquad
O:=J\times V.
$$
Then
$$
O_0\Subset O\Subset\T^n\setminus\overline\omega,
$$
so {\rm(H1)} holds with $O_0$ and $O$ contained in a single falt coordinate
chart.

Since $I$ is a nonempty open interval in $\T$, the one-dimensional
maximal avoiding-geodesic length is finite. Hence
\cite[Theorem~1.3]{Miller2004} gives the required small-time
observability estimate for the one-dimensional Schr\"odinger equation
observed on $I$. By \cite[Theorem~1.4]{Miller2004}, the corresponding
observability estimate holds for the product system on
$\T\times\T^{n-1}$ observed on $I\times\T^{n-1}$. After squaring and
changing the constants, there exist $C_0,T_0>0$ such that
$$
\norm{z}_{L^2(\T^n)}^2
\le
C_0e^{C_0/T}
\int_0^T
\norm{\one_\omega e^{it\Delta}z}_{L^2(\T^n)}^2\,dt,
\qquad
z\in L^2(\T^n),\quad 0<T\le T_0.
$$
For $T>T_0$, the estimate at time $T_0$ remains valid after enlarging
the observation interval. Thus {\rm(H2)} holds. Theorem~\ref{thm:main}
therefore gives \eqref{eq:torus-prefactor}.

Finally, choose $x_1\in\T\setminus\overline I$ and a unit vector
$v'\in\R^{n-1}$. For any $x_0'\in\T^{n-1}$, the unit-speed geodesic
$$
\gamma(t):=(x_1,x_0'+tv'),
\qquad t\ge0,
$$
never meets $\omega=I\times\T^{n-1}$. Hence
$$
L_\omega=+\infty.
$$
This completes the proof.
\end{proof}

\subsection{The whole space}

We next take $\Omega=\R^n$ and
\begin{equation}\label{eq:exterior-ball-control}
\omega:=\mathbb{R}^n \setminus \overline{B_r(x_0)}
=\{x\in\R^n:|x-x_0|>r\},
\qquad r>0.
\end{equation}
Theorem~1.1 of Wang--Wang--Zhang \cite{WangWangZhang2019}, applied with the
same ball at two observation times, yields a constant $C_n>0$ such that, for
all $0\le S<R$ and all $z\in L^2(\R^n)$,
\begin{equation}\label{eq:WWZ-two-time}
\norm{z}_{L^2(\R^n)}^2
\le
C_ne^{C_nr^2/(R-S)}
\left(
\norm{e^{iS\Delta}z}_{L^2(\omega)}^2
+
\norm{e^{iR\Delta}z}_{L^2(\omega)}^2
\right).
\end{equation}

The two-time estimate \eqref{eq:WWZ-two-time} implies the interval
observability estimate required in {\rm(H2)}.

\begin{lemma}\label{lem:whole-space-observability}
Let $\omega$ be given by \eqref{eq:exterior-ball-control}. Then there exists
$C=C(n,r)>0$ such that
\begin{equation}\label{eq:whole-space-C-growth}
\norm{z}_{L^2(\R^n)}^2
\le
Ce^{C/T}
\int_0^T
\norm{\one_\omega e^{it\Delta}z}_{L^2(\R^n)}^2\,dt,
\qquad
z\in L^2(\R^n),\quad T>0.
\end{equation}
In particular, {\rm(H2)} holds.
\end{lemma}
\begin{proof}
First let $0<T\le1$. Take
$S\in(0,T/4)$ and $R\in(3T/4,T)$. Then
\begin{equation}\label{eq:WWZ-time-gap}
R-S\ge\frac T2.
\end{equation}
By \eqref{eq:WWZ-two-time} and \eqref{eq:WWZ-time-gap}, we find
\begin{equation}\label{eq:WWZ-window-pointwise}
\norm{z}_{L^2(\R^n)}^2
\le
C_ne^{2C_nr^2/T}
\left(
\norm{\one_\omega e^{iS\Delta}z}_{L^2(\R^n)}^2
+
\norm{\one_\omega e^{iR\Delta}z}_{L^2(\R^n)}^2
\right).
\end{equation}
Integrating \eqref{eq:WWZ-window-pointwise} first in
$S\in(0,T/4)$ and then in $R\in(3T/4,T)$ gives
\begin{align}
\frac{T^2}{16}\norm{z}_{L^2(\R^n)}^2
&\le
C_ne^{2C_nr^2/T}\frac T4
\left(
\int_0^{T/4}
\norm{\one_\omega e^{it\Delta}z}_{L^2(\R^n)}^2\,dt
\right.
\notag\\
&\qquad\left.
+
\int_{3T/4}^T
\norm{\one_\omega e^{it\Delta}z}_{L^2(\R^n)}^2\,dt
\right)
\notag\\
&\le
C_ne^{2C_nr^2/T}\frac T4
\int_0^T
\norm{\one_\omega e^{it\Delta}z}_{L^2(\R^n)}^2\,dt.
\label{eq:WWZ-integrated}
\end{align}
Therefore
\begin{equation}\label{eq:whole-space-small-time}
\norm{z}_{L^2(\R^n)}^2
\le
\frac{4C_n}{T}e^{2C_nr^2/T}
\int_0^T
\norm{\one_\omega e^{it\Delta}z}_{L^2(\R^n)}^2\,dt,
\qquad 0<T\le1.
\end{equation}
Since $T^{-1}\le e^{1/T}$ for $0<T\le1$, the coefficient in
\eqref{eq:whole-space-small-time} is bounded by $Ce^{C/T}$ after changing
$C=C(n,r)$. For $T>1$, the observability inequality at time $1$ remains
valid after enlarging the observation interval. Increasing $C$ once more
gives \eqref{eq:whole-space-C-growth} for every $T>0$. This completes the proof.
\end{proof}

\begin{corollary}\label{cor:whole-space}
Let $\Omega=\R^n$, and let $\omega$ be given by
\eqref{eq:exterior-ball-control}. Then there exist
$c,C,\delta_0>0$ such that
\begin{equation}\label{eq:whole-space-prefactor}
e^{c\sqrt\delta}
\le
\ChatS(\delta)
\le
e^{C\sqrt\delta},
\qquad \delta\ge\delta_0.
\end{equation}
In particular, \eqref{eq:loglog-limit} holds. 
\end{corollary}

\begin{proof}
By \eqref{eq:exterior-ball-control},
$$
\R^n\setminus\overline\omega=B_r(x_0).
$$
Hence one can choose nonempty open sets
$$
O_0\Subset O\Subset B_r(x_0),
$$
so {\rm(H1)} holds. Lemma~\ref{lem:whole-space-observability} verifies
{\rm(H2)}. The conclusion then follows from Theorem~\ref{thm:main}.
This completes the proof.
\end{proof}

\subsection{Bounded smooth domains}

Let $\Omega\subset\R^n$ be bounded and smooth, and let
$\Delta_\Omega$ be the Dirichlet Laplacian. Quantitative controllability
and observability estimates on bounded domains were studied, among others,
by Phung~\cite{Phung2001}. For the small-time estimate needed here, we use
Miller~\cite{Miller2004}.

Assume, as in \cite{Miller2004}, that generalized geodesics can be uniquely
continued at $\partial\Omega$; this holds, for instance, if
$\partial\Omega$ has no contact of infinite order with its tangent lines.
Let $L_{\Omega,\omega}$ be the supremum of the lengths of generalized
geodesic segments in $\overline\Omega$ which do not meet $\omega$.
If $
L_{\Omega,\omega}<+\infty$,
then \cite[Theorem~1.3]{Miller2004}, after squaring and changing the
constants, gives $C_0,T_0>0$ such that
\begin{equation}\label{eq:miller-domain-small-time}
\norm{z}_{L^2(\Omega)}^2
\le
C_0e^{C_0/T}
\int_0^T
\norm{\one_\omega e^{it\Delta_\Omega}z}_{L^2(\Omega)}^2\,dt,
\qquad
z\in L^2(\Omega),\quad 0<T\le T_0.
\end{equation}
For $T>T_0$, the estimate at time $T_0$ remains valid after enlarging
the observation interval. Hence {\rm(H2)} holds.

\begin{corollary}\label{cor:bounded-domain}
Let $\Omega\subset\R^n$ be a bounded smooth domain, equipped with the
Dirichlet Laplacian, and let $\omega\subset\Omega$ be a nonempty open set
satisfying {\rm(H1)}. Assume that generalized geodesics can be uniquely
continued at $\partial\Omega$ and that
\begin{equation}\label{eq:finite-L-domain}
L_{\Omega,\omega}<+\infty.
\end{equation}
Then there exist $c,C,\delta_0>0$ such that
\begin{equation}\label{eq:bounded-domain-prefactor}
e^{c\sqrt\delta}
\le
\ChatS(\delta)
\le
e^{C\sqrt\delta},
\qquad \delta\ge\delta_0.
\end{equation}
In particular, \eqref{eq:loglog-limit} holds.
\end{corollary}

\begin{proof}
By \cite[Theorem~1.3]{Miller2004} and
\eqref{eq:finite-L-domain}, the observability assumptions
\eqref{eq:observability-inequality}--\eqref{eq:observability-growth} hold.
The conclusion follows from Theorem~\ref{thm:main}. This completes the proof.
\end{proof}

Condition \eqref{eq:finite-L-domain} is the corresponding geometric control
condition for generalized geodesics. Thus Corollary~\ref{cor:bounded-domain}
applies in particular to the standard GCC configurations for bounded smooth
domains.

\section*{Acknowledgments}

S. H. was supported by the National Natural Science Foundation of China under Grants 12171178 and 12171442, and by the Guangdong Basic and Applied Basic Research Foundation under Grant 2026B1515020075.
G. W. was partially supported by the New Cornerstone Science Foundation and the National Natural Science Foundation of China under Grant 12371450.


\begin{thebibliography}{99}



\bibitem{BhandariCapistranoFilhoMajumdarTanaka2024}
K.~Bhandari, R.~de A.~Capistrano-Filho, S.~Majumdar, and T.~Y.~Tanaka,
\newblock Coupled linear Schr\"odinger equations: control and stabilization results,
\newblock \emph{Z. Angew. Math. Phys.} \textbf{75} (2024),
Paper No.~97.

\bibitem{DusserRabah2000}
X.~Dusser and R.~Rabah,
\newblock On exponential stabilizability with arbitrary decay rate for linear
systems in Hilbert spaces,
\newblock \emph{Systems Anal. Model. Simul.} \textbf{37} (2000), no.~4,
417--433.

\bibitem{EngelNagel2000}
K.-J.~Engel and R.~Nagel,
\newblock \emph{One-Parameter Semigroups for Linear Evolution Equations},
\newblock Graduate Texts in Mathematics, vol.~194, Springer-Verlag,
New York, 2000.

\bibitem{FragnelliMoumniSalhi2024}
G.~Fragnelli, A.~Moumni, and J.~Salhi,
\newblock Controllability and stabilization of a degenerate/singular
Schr\"odinger equation,
\newblock \emph{J. Math. Anal. Appl.} \textbf{537} (2024), no.~2,
Paper No.~128290.

\bibitem{Jaffard1990}
S.~Jaffard,
\newblock Contr\^ole interne exact des vibrations d'une plaque rectangulaire,
\newblock \emph{Portugal. Math.} \textbf{47} (1990), no.~4,
423--429.

\bibitem{Komornik1992}
V.~Komornik,
\newblock On the exact internal controllability of a Petrowsky system,
\newblock \emph{J. Math. Pures Appl. (9)} \textbf{71} (1992), no.~4,
331--342.

\bibitem{MaWangYu2023}
Y.~Ma, G.~Wang, and H.~Yu,
\newblock Feedback law to stabilize linear infinite-dimensional systems,
\newblock \emph{Math. Control Relat. Fields} \textbf{13} (2023), no.~3,
1160--1183.

\bibitem{Miller2004}
L.~Miller,
\newblock How violent are fast controls for Schr\"odinger and plate vibrations?,
\newblock \emph{Arch. Ration. Mech. Anal.} \textbf{172} (2004), no.~3,
429--456.

\bibitem{Nguyen2024}
H.-M.~Nguyen,
\newblock Rapid stabilization and finite time stabilization of the bilinear
Schr\"odinger equation,
\newblock to appear in \emph{Ann. Henri Lebesgue},
arXiv:2405.10002, 2024.

\bibitem{Phung2001}
K.-D.~Phung,
\newblock Observability and control of Schr\"odinger equations,
\newblock \emph{SIAM J. Control Optim.} \textbf{40} (2001), no.~1,
211--230.

\bibitem{QuanWang2026}
C.~Quan and G.~Wang,
\newblock Endpoint asymptotics of optimal stabilization prefactors,
\newblock arXiv:2609.05863, 2026.

\bibitem{Treves2022}
F.~Treves,
\newblock \emph{Analytic Partial Differential Equations},
\newblock Grundlehren der mathematischen Wissenschaften, vol.~359,
Springer, Cham, 2022.

\bibitem{Vest2013}
A.~Vest,
\newblock Rapid stabilization in a semigroup framework,
\newblock \emph{SIAM J. Control Optim.} \textbf{51} (2013), no.~5,
4169--4188.

\bibitem{WangWangZhang2019}
G.~Wang, M.~Wang, and Y.~Zhang,
\newblock Observability and unique continuation inequalities for the
Schr\"odinger equation,
\newblock \emph{J. Eur. Math. Soc.} \textbf{21} (2019), no.~11,
3513--3572.

\bibitem{WangXu2017}
G.~Wang and Y.~Xu,
\newblock \emph{Periodic Feedback Stabilization for Linear Periodic Evolution Equations},
\newblock SpringerBriefs in Mathematics,
Springer, Cham, 2017.



\end{thebibliography}
\end{document}